\documentclass[a4paper,11pt]{amsart}

\usepackage{amsfonts}
\usepackage{amssymb}
\usepackage{mathrsfs}
\usepackage{graphics}
\usepackage{tikz}
\usetikzlibrary{matrix,arrows}

\newcommand{\C}{\mathbb{C}}

\newcommand{\Z}{\mathbb{Z}}

\newcommand{\R}{\mathbb{R}}
\newcommand{\tr}{\mathrm{tr\,}}

\newlength{\dhatheight}

\newcommand{\CP}{{\mathbb{C} \mathbb{P}}}

\usepackage[normalem]{ulem}
\usepackage{color}

\def\picill#1by#2(#3)#4
{\vbox to #2
{\hrule width #1 height 0pt depth 0pt
\vfill\special{illustration #3 scaled #4}}}

\newtheorem{theorem}{Theorem}[section]

\newtheorem{lemma}[theorem]{Lemma}

\theoremstyle{definition}

\theoremstyle{remark}
\newtheorem{remark}[theorem]{Remark}

\usepackage{hyperref}

\begin{document}

\title{A remark on first integrals of vector fields}

\author[A.~Belotto da Silva]{Andr\'e Belotto da Silva}
\address{Universit\'e de Paris, Institut de Math\'ematiques de Jussieu Paris Rive Gauche, IMJ-PRG, CNRS 7586, B\^at. Sophie Germain, Place Aur\'elie Nemours, F-75013, Paris, France.}
\email{\href{mailto:belotto@imj-prg.fr}{belotto@imj-prg.fr}}
%\urladdr{\url{https://andrebelotto.com}}

\author[M. Klime\v{s}]{Martin Klime\v{s}}
\address{University of Zagreb, Faculty of Electrical Engineering and Computing, Unska 3, 10000 Zagreb, Croatia.}
\email{\href{mailto:martin.klimes@fer.hr}{martin.klimes@fer.hr}}

\author[J. Rebelo]{Julio Rebelo}
\address{Institut de Math\'ematiques de Toulouse ; UMR 5219, Universit\'e de Toulouse, 118 Route de Narbonne, F-31062 Toulouse, France.}
\email{\href{mailto:rebelo@math.univ-toulouse.fr}{rebelo@math.univ-toulouse.fr}}

\author[H. Reis]{Helena Reis}
\address{Centro de Matem\'atica da Universidade do Porto, Faculdade de Economia da Universidade do Porto, Portugal.}
\email{\href{mailto:hreis@fep.up.pt}{hreis@fep.up.pt}}

%\author{Andr\'e Belotto da Silva, \, \, Martin Klimes, \, \, Julio C. Rebelo \, \, \& \, \, Helena Reis}
%\thanks{}

\subjclass[2010]{Primary 32S65; Secondary 32M25.}
\keywords{Holomorphic vector-field; first integral; formal power series; Stokes phenomena.}

\begin{abstract}
We provide examples of vector fields on $(\C^3, 0)$ admitting a formal first integral but no holomorphic first integral.
These examples are related to a question raised by D. Cerveau and motivated by the celebrated theorems of Malgrange \cite{Malg} and Mattei-Moussu \cite{MM}.
%Motivated by the local analysis of a saddle-node singularity appearing in a convenient birational model for Airy equation, 
\end{abstract}

\maketitle
\setcounter{tocdepth}{1}

\section{Introduction}
A celebrated theorem due to Mattei and Moussu \cite{MM} states that a holomorphic codimension~$1$ foliation admitting a formal first integral necessarily possesses a holomorphic first integral as well. The theorem and its proof completely clarify the relationship between formal and holomorphic first integrals for codimension~$1$ foliations, whereas the general investigation of the existence of these first integrals also includes an influential work of Malgrange \cite{Malg1}. For higher codimension foliations, the relationship between formal and holomorphic first integrals remains quite mysterious. In this context, D. Cerveau naturally asked whether a holomorphic vector field $X$ defined on a neighborhood of the origin of $\C^3$ and admitting one - or two - formal first integral must possess holomorphic first integrals as well. The goal of this paper is to show that the existence of a single
formal first integral is not
enough to guarantee the existence of holomorphic ones and this will be done by means of the following theorem:

%(that is, there exists $F\in \C[[x]]$ such that $\omega \wedge dF \equiv 0$)

%In the later, Malgrange proves that if $\omega$ admits a strong formal first integral (that is, there exists $F,\,G\in \C[[x]]$ such that formally $\omega =G\,  dF $), then it admits an holomorphic (strong) first integral.
%. It can be seen as an improvement of a previous result of Malgrange \cite{Malg1}, which states that if $\omega$ admits a strong formal first integral (that is, there exists $F,\,G\in \C[[x]]$ such that $\omega =G\,  dF $), then it admits an holomorphic (strong) first integral.

%In this context, Malgrange extended his result about strong first integrals to codimension~$r$ foliations \cite{Malg}, under the hypothesis of existence of $r$ ``independent" formal first integrals, that is, the existence of $r$ first integrals $F_1,\ldots,F_r \in \C[[x]]$ such that $dF_1 \wedge \cdots \wedge dF_r \not\equiv 0$. 
%Note that in the later work it is necessary to assume the existence of at least $r$ ``independent" formal first integrals, that is, the existence of $r$ first integrals $F_1,\ldots,F_r \in \C[[x]]$ such that $dF_1 \wedge \cdots \wedge dF_r \not\equiv 0$.
%Based on these considerations, 

\begin{theorem}
\label{Cerveau_theexample}
Consider the family $X_{a,b,c}$ of vector fields on $\C^3$ defined by
\begin{equation}\label{eq:vf}
X_{a,b,c}= x^2 \frac{\partial}{\partial x} + (1+ax)\left[ y_1\frac{\partial}{\partial y_1}-
y_2 \frac{\partial}{\partial y_2} \right] +  bxy_2\frac{\partial}{\partial y_1} + cxy_1\frac{\partial}{\partial y_2} \, ,
\end{equation}
where $a$, $b$, and $c$ are complex parameters. Assume that the parameters  are such that
\[
\cos (2\pi a) \neq \cos (2\pi\sqrt{a^2+bc}) \, .
\]
Then the vector field $X_{a,b,c}$ possesses no (non-constant) holomorphic first integrals, albeit it does
possess formal first integrals.
\end{theorem}

In particular, the vector field $X_{1,1,1}$ obtained by setting $a=b=c=1$ admits a formal first integral but no holomorphic one.
The existence of these examples was certainly expected, given the transcendental behavior of singular foliations, but we highlight the
simplicity of its expression which suggests that this might be a fairly common phenomena in applications. The issue is therefore also
related to Malgrange's theorem in \cite{Malg} in that they confirm that some (strong) additional assumptions are, in fact, needed
(see below for further information).
As a side note, the simple nature of the examples provided by Theorem~\ref{Cerveau_theexample} also bears some similarities with
certain results quoted in the survey article \cite{stolovitch} by Stolovitch: many normalization results for ``simple'' vector fields
having formulas not too different
from $X_{a,b,c}$ are presented under some additional geometric condition (for example ``volume-preserving'' or ``hamiltonian''). We might wonder
what is the influence of these conditions on the problem discussed here. Conversely, it is also fair to wonder if
the (potential) ability to turn formal first integrals into holomorphic ones may shed new light in more general normal form problems.

As a matter of fact, our observation of the vector field \eqref{eq:vf} is a by-product of our study of the global dynamics of the Airy and Painlev\'{e} I and II equations \cite{BKRR}. Our original motivation was the local analysis of the saddle-node singularity associated with the vector-field 
\[
Y_A= - \frac{1}{2} x^4 \frac{\partial}{\partial x} + \left(z-\frac{1}{2}x^3y\right)
\frac{\partial}{\partial y} + (y- x^3z) \frac{\partial}{\partial z} \,,
\]
which appears in a convenient birational model for the compactified Airy equation. The formal normal form of $Y_A$ as well as the
corresponding Stokes phenomenon can accurately be computed with the same technique detailed in Section~2 for the
vector field $X_{a,b,c}$. In doing so, there follows
that $Y_A$ admits a first integral in the field of fractions of formal power series, i.e., there is a formal first integral of the form
$F/G$ with $F,\,G \in \C [[ x,y,z ]]$. Yet $Y_A$ has no holomorphic or meromorphic first integral.
Basically, the difference between the example provided by $Y_A$
and Cerveau's general questions lies in the fact the ``formal first integral'' of the vector field $Y_A$ has a ``meromorphic'' nature rather than a more standard power series representation without negative terms. In turn, there are deep differences between first integrals of ``holomorphic'' and of ``meromorphic''
natures as already underlined in the topological context. In fact, in codimension~$1$, Mattei-Moussu \cite{MM} theorem asserts that first integrals
are topological invariants and the existence of formal first integrals implies the existence of holomorphic ones. On the other hand, the existence of
meromorphic first integrals is not a topological invariant already in the two-dimensional ambient case, cf. \cite{Cerveau,Martine,P-R}.
Similarly, in codimension~$2$ complete integrability in the holomorphic sense is not a topological invariant either \cite{P-R}.
From this point of view, the vector-field $Y_A$ falls genuinely short of shedding light into Cerveau's questions due to the nature of its formal
first integral.

It is now interesting to investigate whether a holomorphic vector field $X$ defined on a neighborhood of the origin of $\C^3$ and admitting two formal first integrals $F_1,\,F_2$ such that $dF_1 \wedge dF_2 \not\equiv 0$, necessarily admits at least one holomorphic first integral. The best result in this direction, as far as we are aware of, remains the previously mentioned
theorem of Malgrange \cite{Malg} concerning Pfaffian systems in arbitrary dimensions. More precisely, given a codimension $r$ foliation defined in some open set of $\mathbb{C}^n$ and generated by $r$ one-forms $\Omega = \{\omega_1,\ldots,\omega_r\}$, denote by $S(\Omega)$ the singular locus of $\Omega$, that is, the set of points where the $r$-form $\omega_1 \wedge \ldots \wedge \omega_r$ is identically zero. We say that $\Omega$ is integrable (respectively formally integrable) at $x\in \mathbb{C}^n$, if there exists $r$ holomorphic function germs $f_1,\ldots,f_r \in \mathcal{O}_x$ (respectively $r$ formal power series in $ \widehat{\mathcal{O}}_x$) such that the module generated by $\{df_1, \ldots,df_r\}$ coincides with $\Omega \cdot \mathcal{O}_{x}$ (respectively, with $\Omega \cdot \widehat{\mathcal{O}}_{x}$). In \cite{Malg}, Malgrange shows that if $S(\Omega)$ has codimension $3$, or if $\Omega$ is formally integrable and $S(\Omega)$ has codimension $2$, then $\Omega$ is integrable. As mentioned, these hypotheses are generally quite strong when we consider a Pfaffian system obtained as the dual of a vector field.

The proof of Theorem~\ref{Cerveau_theexample} relies on the standard theory of linear systems (normal forms and Stokes phenomena among others).
We refer the reader to \cite[$\S$16 and 20]{IY} and references there-within (or \cite{Balser,Wasow}) for an
introduction to the methods used in this work.

%We will provide more details in \cite{BKRR}.

%In this context, Malgrange extended his result about strong first integrals to codimension~$r$ foliations \cite{Malg}, under the hypothesis of existence of $r$ ``independent" formal first integrals, that is, the existence of $r$ first integrals $F_1,\ldots,F_r \in \C[[x]]$ such that $dF_1 \wedge \cdots \wedge dF_r \not\equiv 0$. 
%Note that in the later work it is necessary to assume the existence of at least $r$ ``independent" formal first integrals, that is, the existence of $r$ first integrals $F_1,\ldots,F_r \in \C[[x]]$ such that $dF_1 \wedge \cdots \wedge dF_r \not\equiv 0$.
%Based on these considerations, 

%It is somehow striking that Theorem \ref{Cerveau_theexample} was not previously known. The existence of such an example was certainly expected given the transcendental behavior of singular foliations, but might have been missed in previous works. 

\medskip
\noindent
\textbf{Acknowledgment.} H. Reis was
partially supported by CMUP, which is financed by national funds through FCT – Fundação para a Ciência e Tecnologia, I.P.,
under the project with reference UIDB/00144/2020. J. Rebelo and H. Reis are also partially supported by CIMI through the project
``Complex dynamics of group actions, Halphen and Painlev\'e systems''.

\section{Proof of Theorem {\ref{Cerveau_theexample}}}\label{sec:2}

Let us begin our approach to Theorem~\ref{Cerveau_theexample} by noticing that the vector field
$X_{a,b,c}$ is associated with the time-dependent linear differential system
\begin{equation}\label{eq:system}
x^2\frac{dy}{dx}=
\begin{bmatrix} 1+ax & bx \\
cx & -1-ax
\end{bmatrix}
y,\qquad y=
\begin{bmatrix}
y_1 \\ y_2
\end{bmatrix} \, .
\end{equation}
Following classical terminology of linear systems,
the system above has a non-resonant irregular singular point of Poincar\'e rank 1 at $x=0$, see for example \cite[$\S$~20]{IY}.
Note that this differential system is in the so-called Birkhoff normal form: the system is well defined for all $x\in\CP^1$ and
has only two singular points, namely $x=0$ and $x = \infty$, see e.g. \cite[$\S$~20B]{IY}.
In turn, the singularity at $x=\infty$ is a {\it Fuchsian one}. In other words, the system has a simple pole at $x=\infty$,
see e.g. \cite[Definition~16.9]{IY}.
In addition, since the linear system~(\ref{eq:system}) is non-resonant, it can formally be transformed into a diagonal linear system by means of the
standard Poincar\'{e}-Dulac method \cite[Theorem~20.7]{IY}. Whereas the resulting (formal) power series is divergent, Sibuya's Theorem
asserts that it is Borel 1-summable in all
directions $x \in e^{i\alpha} \R_{>0}$ with exception of the singular directions, namely the directions corresponding to $\alpha \in \pi \Z$.
The preceding is made accurate by the lemma below:

\begin{lemma}\label{lem:ChangeCoordinates}
There exists a formal linear change of coordinates having the form $y=\hat T(x)u$, with $\hat T(0)=I$, which conjugates system~(\ref{eq:system})
to the (diagonal) linear system
\begin{equation}\label{eq:diagonalsystem}
x^2\frac{du}{dx}=\begin{bmatrix} 1+ax & 0 \\ 0 & -1-ax \end{bmatrix}u, \qquad u=
\begin{bmatrix}
u_1 \\ u_2
\end{bmatrix}.
\end{equation}
Moreover, for every $\alpha \in\ ]0, \pi[\ \cup\ ]\pi, 2\pi[$, there exists a holomorphic transformation $y = T_\alpha(x)u$ conjugating
systems~(\ref{eq:system}) and~(\ref{eq:diagonalsystem}) and satisfying the following conditions:
\begin{itemize}
\item[(a)] $T_\alpha(x)$ is analytic on the open sector of opening angle $\pi$ bisected by the half-line
$e^{i\alpha}\R_{>0}$.

\item[(b)] $T_\alpha(x)$ and $T_\beta(x)$, with $\alpha < \beta$, coincide on the intersection of the corresponding half planes provided that the interval $]\alpha, \beta[$ does not contains an integral multiple of $\pi$.

\item[(c)] $T_\alpha(x)$ is asymptotic to $\hat T(x)$.
\end{itemize}

\end{lemma}
\begin{proof}
As previously stated, the existence of a formal change of variables conjugating systems~(\ref{eq:system}) and~(\ref{eq:diagonalsystem}),
as well as its analytic nature on the indicated sectors, goes back to classical results by Birkhoff and Malmquist (or more general versions
by Hukuhara, Turittin, and Sibuya, see \cite[Theorems~20.7 and~20.16]{IY}).
Therefore it only remains to check that the diagonal matrix appearing in~(\ref{eq:diagonalsystem}) has the indicated form.
To do this,
note that the formal invariants of the initial system~(\ref{eq:system}) can be read off a suitable finite jet of the eigenvalue functions
associated with the matrix
\[
\begin{bmatrix} 1+ax & bx \\
cx & -1-ax
\end{bmatrix}.
\]
Clearly these eigenvalue functions are equal to
$ \pm \sqrt{\left(1+ax\right)^2 + bcx^2}$. Now, since the Poincar\'{e} rank of the singularity is $1$, only the 1-jet of the eigenvalue function
is a formal invariant, c.f. \cite[Proposition 20.2]{IY}. Therefore $\pm(1+ax)$ are the only formal invariants of the system. This completes
the proof of the lemma.
\end{proof}

Next, note that the system~(\ref{eq:diagonalsystem}) clearly admits
\begin{equation}\label{eq:fundamentalSolution}
U(x)=\begin{bmatrix} x^a e^{-1/x} & 0 \\ 0 &  x^{-a} e^{1/x} \end{bmatrix} 
\end{equation}
as fundamental (matrix) solution.
Consider one of the two singular directions, namely $\beta = 0$ or $\beta = \pi$. Let $T_{\beta+}$ and
$T_{\beta-}$ denote, respectively, the Borel sums on the ``left'' and on the ``right'' of the fixed singular direction $\beta$.
Then, there is a constant matrix $S_{\beta}$ satisfying
\[
T_{\beta-}(x)=T_{\beta+}(x)  U(x)S_\beta U(x)^{-1} \, ,
\]
for $x\in e^{i\beta}\R_{>0}$.
The matrices $S_0$ and $S_{\pi}$ are called the {\it Stokes matrices}\, and they have the general forms
\[
S_0=\begin{bmatrix}
1 & s_0 \\
0 & 1
\end{bmatrix} \qquad \text{and} \qquad S_\pi = \begin{bmatrix}
1 & 0 \\
s_{\pi} & 1
\end{bmatrix} \, ,
\]
for suitable constants $s_0, s_{\pi} \in \C$, see e.g. \cite[$\S$~20G]{IY}. In the particular case in question, explicit formulas
for $s_0$ and $s_{\pi}$ are known, see \cite[pages 86 and 87]{BJL}.
However, for our purposes, it suffices to prove that:

\begin{lemma}\label{lem:NonTrivialStokes}
The product $s_0 s_\pi \neq 0 $ if and only if $\cos (2\pi a) \neq \cos (2\pi\sqrt{a^2+bc})$.
\end{lemma}
\begin{proof}
The lemma will be proved by explicitly computing the monodromy matrix $M$ associated with the system~(\ref{eq:system}) around $x=0$
in two different ways: first we compute the matrix directly around $x=0$ by using the Stokes matrices and then we will compute the monodromy (holonomy)
around $x=\infty$ which is a Fuchsian singular point. The monodromy around $x=\infty$ is the inverse of the monodromy matrix $M$ since the system
in question has only two singular points (corresponding to $x=0$ and to $x = \infty$). The result will then easily follow by computing the trace
of $M$ in each situation.

\vspace{0.1cm}

\noindent {\it Claim}. The monodromy matrix around the origin is conjugate to $M = S_0 N S_\pi$, where
\[
N=\begin{bmatrix} e^{2\pi i a} & 0 \\ 0 &  e^{-2\pi i a} \end{bmatrix}
\]
is the ``formal monodromy'' of the fundamental matrix solution $U(x)$ introduced in Equation~(\ref{eq:fundamentalSolution}).

\noindent {\it Proof of the claim}. The statement follows from the sequence
of equations
\begin{align*} 
T_{0+}(e^{2\pi i}x)U(e^{2\pi i}x)  &= T_{\pi-}(e^{2\pi i}x) U(e^{2\pi i}x) = T_{\pi+}(e^{2\pi i}x) U(e^{2\pi i}x) S_\pi \\
&=T_{2\pi-}(e^{2\pi i}x) U(e^{2\pi i}x) S_\pi = T_{0-}(x) U(x) N S_\pi = T_{0+}(x) U(x) S_0 N S_\pi \, , 
%\red{\sout{ = Y_{0+}(x) S_0 N S_\pi}}
\end{align*}
where item~(b) of Lemma~\ref{lem:ChangeCoordinates} has implicitly been used.\qed

Since $M = S_0 N S_\pi$, it immediately follows that 
\[
\tr M = 2\cos (2\pi a)+e^{-2\pi i a}s_0s_\pi.
\]

Let us now compute the matrix $M$ by looking at the singular point $x =\infty$. Let $v = 1/x$ so that the system~(\ref{eq:system}) becomes
\begin{equation}\label{eq:system_infty}
v \frac{dy}{dv} = \begin{bmatrix} a + v & b \\ c & -a - v \end{bmatrix}y  = A(v) y\,,
\end{equation}
and note that $v=0$ corresponds to $x =\infty$. Denote by
$\lambda_1$ and $\lambda_2$ the eigenvalues of of the matrix $A(0)$. Naturally the matrix $A(0)$ is the so-called
{\it residue matrix}\, of system~(\ref{eq:system_infty}). Clearly these two eigenvalues are symmetric and, up to relabeling,
we set $\lambda_1 = \lambda$ and $\lambda_2 = -\lambda$ where $\lambda =\sqrt{a^2+bc}$.
In this case, the system is non-resonant if $2 \lambda \notin \mathbb{Z}$, see e.g. \cite[Definition~16.12]{IY}. In turn, provided that there
is no resonance, the system is locally holomorphically equivalent to the Euler system $tv' = A(0)v$, see e.g. \cite[Theorem~16.16]{IY}.
In turn, the monodromy matrix around $v=0$ is conjugate to the exponential of $2\pi i A(0)$ and the latter matrix is conjugate to the inverse of the initial
monodromy matrix $M$. Since traces of matrices remain invariant under conjugations, the preceding finally yields
\[
 \tr M = 2\cos (2 \pi \lambda) \, .
\]
In fact, this last formula holds whether or not the system~(\ref{eq:system_infty}) is resonant as it immediately follows 
from the continuity of $\tr M$ with respect to the parameters $a$, $b$, and $c$ (the set of non-resonant systems is open and dense).
Lemma~\ref{lem:NonTrivialStokes} promptly follows.
\end{proof}

%We postpone the details of the proof of the Lemma to subsection \ref{ssec:NonTrivialStokes}. 

Next, note that the the diagonal differential system~(\ref{eq:diagonalsystem}) is naturally equivalent to the following family of vector fields on $\C^3$:
\begin{equation}\label{eq:vf0}
X_a = x^2 \frac{\partial}{\partial x} + (1+ax)\left[ u_1 \frac{\partial}{\partial u_1}
-u_2 \frac{\partial}{\partial u_2} \right] \, ,
\end{equation}
where $a \in \C$. Clearly vector fields in the family $X_a$ admits the function $h(u_1,u_2) = u_1 u_2$ as a
holomorphic first integral. Furthermore, we have:

\begin{lemma}\label{lemma:FirstIntegrals}
The function $h(u) = u_1u_2$ is a primitive first integral of $X_a$ in the following sense: if $\hat F = \hat F(x,u_1,u_2) \in \C [[ x,u ]]$
is a formal first integral of $X_a$, then there exists a formal power series $\hat G \in \mathbb{C} [[ z ]]$ such that $\hat F =\hat G \circ h$. 
\end{lemma}

\begin{proof}
Assume that $\hat F = \hat F(x,u_1,u_2)$ is a formal
first integral of $X_a$ and consider a Taylor expansion of the form:
\begin{equation}
\hat F(x,u_1,u_2) = \sum_{j=0}^{\infty} x^j \hat f_j(u_1,u_2) = \sum_{j=0}^{\infty} x^j \sum_{k \in \mathbb{Z}}u_1^k \hat f_{j,k}(u_1u_2) \, . \label{weirdTaylorexpansion}
\end{equation}

\noindent {\it Claim}. We have $\hat F(x,u_1,u_2) = \hat f_{0,0}(u_1u_2) + o(x^n)$ for all $n\in \mathbb{N}$.

Clearly the lemma is an immediate consequence of the claim so that it suffices to prove the claim.

\noindent {\it Proof of the Claim}. We argue by induction. Assume the claim holds for $n=n_0$ (where the possibility of having $n_0 =0$ is not
excluded). Since $F$ is a formal first integral of $X_a$, a direct computation yields
\[
0 = d_0\hat F.X_a = x^{n_0}\left( \sum_{k \in \mathbb{Z}}k u_1^k \hat f_{n_0,k}(u_1u_2)    \right) + o(x^{n_0+1}) \, .
\]
By comparing monomial degrees, there follows that all the functions $\hat f_{n_0,k}(\cdot)$ must vanish identically provided that $k \neq 0$.
Thus the power series expansion~(\ref{weirdTaylorexpansion}) of $\hat F$ takes on the form
\[
\hat F = \hat f_{0,0} (u_1 u_2) + x^{n_0} \hat f_{n_0,0} (u_1u_2) + \sum_{j=n_0+1}^{\infty} x^j \sum_{k \in \mathbb{Z}}u_1^k \hat f_{j,k}(u_1u_2) \, .
\]
In turn, this refined formula for $\hat F$ yields
\[
0 = d_0\hat F.X_a = x^{n_0+1}\left(n_0\hat f_{n_0,0}(u_1u_2) + \sum_{k \in \mathbb{Z}}k u_1^k \hat f_{n_0+1,k}(u_1u_2)\right)+ o(x^{n_0+2}) \, .
\]
Therefore also $\hat f_{n_0,0}(\cdot)$ must vanish identically unless $n_0 =0$. Hence $\hat F$ is actually of the form $\hat F(x,u_1,u_2) = \hat f_{0,0}(u_1u_2) +
o(x^{n_0+1})$ which establishes the induction step. The proof of the claim is complete and so is the proof of the lemma.
\end{proof}

\begin{remark}
The computation carried out in the proof of Lemma~\ref{lemma:FirstIntegrals} is related to a qualitative issue that is worth pointing out.
For this, note first that
the general solution of the diagonal system~(\ref{eq:diagonalsystem}) has the form
\[
\begin{cases}
u_1 (x) = c_1 e^{-1/x}x^a \\
u_2 (x) = c_2 e^{1/x}x^{-a}
\end{cases}
\]
for suitable constants $c_1, \, c_2 \in \C$. It follows that for any formal first integral $\hat F = \hat F(x,u_1,u_2)$ of $X_a$, the composition $\hat F(x,c_1e^{-\frac{1}{x}}x^a,c_2e^{\frac{1}{x}}x^{-a})$ must be a constant, and hence must factor through $h$ due to the presence of the essential
singular point arising from $e^{\frac{1}{x}}$. 
\end{remark}
%We postpone the details of the proof of the Lemma to subsection \ref{ssec:FirstIntegrals}. 

We are now able to provide the proof of Theorem~\ref{Cerveau_theexample}.

\begin{proof}[Proof of Theorem~\ref{Cerveau_theexample}]
Owing to Lemma~\ref{lem:ChangeCoordinates}, the vector field $X_{a,b,c}$ has a {\it formal first integral} $\hat f = \hat f (x,y_1,y_2)$
which is obtained out of the first integral $h(u_1,u_2)$ of $X_a$ by means of the equation
\[
\hat f(x,\hat T(x)u) = h(u) \, . %\label{equation_definingfirstintegral-f}
\]
Furthermore, according to Lemma~\ref{lemma:FirstIntegrals}, every other
formal first integral of $X_{a,b,c}$ must formally factor through $\hat f$. The proof of Theorem~\ref{Cerveau_theexample} is
then reduced to showing that if a (non-constant) first integral of $X_{a,b,c}$ is holomorphic then we necessarily
have $\cos (2\pi a) = \cos (2\pi\sqrt{a^2+bc})$.

Let us then assume there is a (non-constant) holomorphic first integral $f$ for the vector field $X_{a,b,c}$ defined as in~(\ref{eq:vf}).
It follows from Lemmas \ref{lem:ChangeCoordinates} and \ref{lemma:FirstIntegrals} the existence of a formal series $\hat G \in \C [[ z ]]$
such that $f(x,\hat{T}(x)u) = \hat G \circ h$, where $u=(u_1,u_2)$. Since the formal series on the left side 
is $1$-Borel summable in the variable $x$, while the right hand side is independent of $x$,
we conclude that $\hat G \circ h$ is an analytic function on $u = (u_1, u_2)$. We set $\hat G \circ h = g(u)$.
In particular, there follows that $f_\alpha(x,T_\alpha(x)u) = g(u)$, where $f_\alpha(x,y)$ denotes a sectorial Borel sum.
We thus obtain that $f_{\beta+} = f_{\beta-}$ for both singular directions $\beta=0$ and $\beta=\pi$, where $f_{\beta+}$ (resp. $f_{\beta-}$) stands
as usual for the Borel sum on the ``left'' (resp. ``right'') of the fixed singular
direction $\beta$. Therefore we have
\begin{align*}
	g(u)= & f_{\beta-}(x,T_{\beta-}(x)u) = f_{\beta+}(x,T_{\beta+}(x)U(x)S_\beta U(x)^{-1}u) \\
        = &g(U(x)S_\beta U(x)^{-1}u) \, .
\end{align*}
In other words, the function $g(u)$ is invariant by the Stokes operators
\[
S: \, u \mapsto U(x) S_\beta U(x)^{-1} u \, .
\]
However, it follows from direct computation that the function $g(u)$, which factors through $h(u)=u_1u_2$, is invariant by this
operator only if $s_{\beta} =0$. Hence, the existence of the holomorphic first integral $f$ implies that the product $s_0 s_\pi$ equals zero
so that Lemma \ref{lem:NonTrivialStokes} ensures that we must have $\cos (2\pi a) = \cos (2\pi\sqrt{a^2+bc})$. This ends the
proof of Theorem \ref{Cerveau_theexample}. 
\end{proof}

\end{document}